\documentclass[11pt,a4paper,reqno]{amsart}

\usepackage{graphicx,mathrsfs,amssymb}
\usepackage[usenames]{color}
\usepackage[colorlinks=true,citecolor=blue]{hyperref}
\usepackage{enumerate}

\newtheorem{theorem}{Theorem}[section]
\newtheorem{lemma}[theorem]{Lemma}

\theoremstyle{definition}
\newtheorem{definition}[theorem]{Definition}
\theoremstyle{remark}
\newtheorem{remark}[theorem]{Remark}

\numberwithin{equation}{section}

\newcommand{\Z}{{\mathbb Z}}
\newcommand{\T}{{\mathbb T}}

\newcommand{\N}{{\mathbb N}}
\newcommand{\R}{{\mathbb R}}
\newcommand{\bk}{\mathbf{k}}

\newcommand{\dbar}{{\textstyle d^{\hspace*{-0.4em}-}\hspace*{-0.15em}}}

\newcommand{\supp}{\mathop{\textrm{supp}}}

\newcommand{\op}{\mathop{\textrm{\upshape op}}}

\newcommand{\id}{\mathop{\textrm{\upshape id}}\nolimits}

\renewcommand{\epsilon}{\varepsilon}

\renewcommand{\tilde}[1]{\widetilde{#1}}

\allowdisplaybreaks
\begin{document}

\title[Operator-valued pseudodifferential operators on the torus]{Mapping properties for  operator-valued pseudodifferential operators on toroidal Besov spaces}


\author[B. Barraza Mart\'{i}nez]{Bienvenido Barraza Mart\'{i}nez}
\address{B. Barraza Mart\'{i}nez, Universidad del Norte, Departamento de Matem\'aticas,  Barranquilla (Colombia)}
\email{bbarraza@uninorte.edu.co}

\author[R. Denk]{Robert Denk}
\address{R.\ Denk, Universit\"at Konstanz, Fachbereich f\"ur Mathematik und Statistik, Konstanz (Germany)}
\email{robert.denk@uni-konstanz.de}

\author[J. Hern\'andez Monz\'on]{Jairo Hern\'andez Monz\'on}
\address{J. Hern\'andez Monz\'on, Universidad del Norte, Departamento de Matem\'aticas,  Barranquilla (Colombia)}
\email{jahernan@uninorte.edu.co}

\author[M. Nendel]{Max Nendel}
\address{M.\ Nendel, Universit\"at Konstanz, Fachbereich f\"ur Mathematik und Statistik, Konstanz (Germany)}
\email{max.nendel@uni-konstanz.de}

\thanks{The authors would like to thank COLCIENCIAS (Project 121556933488) and DAAD for the financial support.}

\begin{abstract}
In this paper, we consider  pseudodifferential operators on the torus with
operator-valued symbols and prove continuity properties on vector-valued toroidal
Besov spaces, without assumptions on the
underlying Banach spaces. The symbols are of limited smoothness with respect to $x$ and satisfy a finite number of estimates on the discrete derivatives. The proof of the main result is based on a description of the operator as a convolution operator with a kernel representation which is related to the dyadic decomposition appearing in the definition of the Besov space.
\end{abstract}



\

\date{June 22, 2017}
\maketitle

\section{Introduction}

In this note, we consider mapping properties of pseudodifferential operators on the $n$-dimensional torus $\T^n = (\R/2\pi\Z)^n$ in vector-valued Besov spaces. Toroidal pseudodifferential operators are defined and investigated, e.g., in the monograph \cite{Ruzhansky-Turunen10} by Ruzhansky and Turunen. Here, the group structure of $\T^n$ is used to define a global quantization with covariable $\bk \in\Z^n$ (Fourier series). This quantization is also the basis for the definition of the Besov spaces on the torus by means of a dyadic decomposition of $\Z^n$ (see Definition~\ref{2.5} below). Compared to the other possible approach where $\T^n$ is treated as a closed manifold, one has the advantage of a global quantization without the necessity to introduce local coordinate charts. The theory of pseudodifferential operators on the torus was developed by Agranovich \cite{agranovich79}, McLean \cite{mclean91}, Melo \cite{melo97}, Bu-Kim \cite{Bu-Kim04}, \cite{Bu-Kim05a} and others.

Mapping properties of toroidal pseudodifferential operators in $L^p$-spaces were studied studied by Delgado \cite{delgado13}, Molahajloo-Shahla-Wong \cite{molahajloo-shahla-wong10}, Wong \cite{wong11}, Cardona \cite{Cardona17a} and others. In particular, in Cardona \cite{Cardona17a} mapping properties in Besov and H\"older spaces are shown. The global quantization approach mentioned above can be generalized to compact Lie groups, see Ruzhansky-Turunen \cite{Ruzhansky-Turunen13}, Ruzhansky-Turunen-Wirth \cite{Ruzhansky-Turunen-Wirth14}, Cardona \cite{Cardona17} and references therein.

The above references deal with  the scalar-valued case. In the situation where the considered functions have values in some Banach space $E$, the situation depends on the geometric properties of $E$. If $E$ is a UMD space (and hence in particular reflexive), then Mikhlin-type results yield $L^p$-boundedness, see Arendt-Bu \cite{Arendt-Bu04}, Keyantuo-Lizama-Poblete \cite{Keyantuo-Lizama-Poblete09},  Barraza-Gonz\'alez-Hern\'andez \cite{Barraza-Gonzalez-Hernandez15}. The case of general Banach spaces was studied by Amann  \cite{Amann97} on $\R^n$ and by Denk-Barraza-Hern\'{a}ndez-Nau \cite{Barraza-Denk-Hernandez_Nau16} on $\T^n$. While in \cite{Barraza-Denk-Hernandez_Nau16} only pseudodifferential operators with  $x$-independent symbols (Fourier multipliers) were studied, in the present note we investigate $x$-dependent vector-valued symbols with values in a general Banach space.

We consider pseudodifferential operators whose symbols have limited smoothness with respect to $x$ and satisfy a finite number of growth conditions in analogy to the conditions of H\"ormander. The symbols have values in $L(E)$, the space of all bounded linear operators in $E$, where $E$ stands for an arbitrary Banach space. The main result (Theorem~\ref{3.3}) states that the pseudodifferential operator $\op[a]$ related to the symbol $a$ of order $m$ induces a bounded linear operator from $B_{pq}^{s+m}(\T^n, E)$ to $B_{pq}^s(\T^n,E)$, where the range of $s$ is in a natural way restricted by the smoothness of $a$ and where $p,q\in[1,\infty]$. One of the main steps in the proof consists of a description of the operators $\op[a]\op[\phi_j]$ and $\op[\phi_j]\op[a]$ as convolution operators (see Lemma~\ref{2.6}). Here $(\phi_j)_{j\in\N_0}$ is a dyadic decomposition of $\Z^n$, and the kernels of these operators can be written in form of an infinite sum adapted to this dyadic decomposition. This allows  to avoid oscillatory integrals and sum-integrals. We note that this approach gives a new proof of the Besov space continuity even in the $x$-independent case (cf. \cite{Barraza-Denk-Hernandez_Nau16}), and therefore it may serve as a basis for future generalizations to locally compact abelian groups and to compact Lie groups (see also Remark~\ref{3.4} a)). Both the mapping properties and the convolution kernel description can be used to show generation of analytic semigroups for parabolic pseudodifferential operators on the torus. This will be the content of a subsequent paper.

\section{Kernel estimates for toroidal pseudodifferential operators}

In the following, let $E$ be a Banach space with norm $\|\cdot\|$. Throughout this paper, we fix $n\in\N$, $\rho\in\N$ with $\rho\ge n+1$, $r\in [0,\infty)$ and $m\in\R$. We consider operator-valued pseudodifferential operators on the $n$-dimensional torus $\T^n = (\R/(2\pi\Z))^n$, where we use $[-\pi,\pi]^n$ as a set of representatives. Note that in this case, the euclidian norm $|x|$ of a representative equals the distance of $x$ to $0$ in the metric on $\T^n$. We use standard notation for smooth vector-valued functions $f\in C^\infty(\T^n,E)$ and their Fourier series (discrete Fourier transform)
\[ (\mathscr Ff)(\bk) := \hat f(\bk) := \int_{\T^n} e^{-ik\cdot x} f(x)\dbar x \quad (\bk\in\Z^n),\]
where $\dbar x := (2\pi)^{-n} dx$. The Fourier transform is extended by duality to the space of vector-valued toroidal distributions $u\in \mathscr D'(\T^n,E) := L(C^\infty(\T^n),E)$, see \cite{Barraza-Denk-Hernandez_Nau16}, Section~2 for more details.

The symbol class on the torus is defined with help of the discrete derivatives (differences) $\Delta_\bk^\alpha$. For this, let $j\in\{1,\dots,n\}$, and let $\delta_j :=(\delta_{jk})_{k=1,\dots,n}$ be the $j$-th unit vector in $\R^n$. For $a\colon\Z^n\to E$ and $\alpha\in\N_0^n$, we set
\begin{align*}
  \Delta_{k_j} a(\bk) & := a(\bk +\delta_j) - a(\bk) \quad (\bk\in\Z^n),\\
  \Delta_{\bk}^\alpha & := \Delta_{k_1}^{\alpha_1}\ldots\Delta_{k_n}^{\alpha_n}.
\end{align*}
We refer to \cite{Ruzhansky-Turunen10}, Sect. 3.3.1, for a more detailed discussion of the discrete analysis on the torus. In the following definition, we set $\langle\bk\rangle := (1+|\bk|^2)^{1/2}\;(\bk\in\Z^n)$.

\begin{definition}\label{2.1}
a) Let $S^{m,\rho,r}:=S^{m,\rho,r}\left(  \mathbb{T}%
^{n}\times\mathbb{Z}^{n},L(E)\right)  $ be the set of all
functions $a:\mathbb{T}^{n}\times\mathbb{Z}^{n}\longrightarrow L(E)$ such that
$\left[  x\longmapsto a\left(  x,\mathbf{k}\right)  \right]  \in C^{r}\left(
\mathbb{T}^{n},L(E)\right)  $ for all $\mathbf{k}\in\mathbb{Z}^{n}$, and
$\left\Vert a\right\Vert _{m}^{\left(  \rho,r\right)  }<\infty$. Here, in the case $r\in\N_0$ we
define
\[
\left\Vert a\right\Vert _{m}^{\left(  \rho,r\right)  }:=\max_{\left\vert
\alpha\right\vert \leq\rho}\max_{\left\vert \beta\right\vert \leq r}\sup
_{x\in\mathbb{T}^{n}}\sup_{\mathbf{k}\in\mathbb{Z}^{n}}\left\langle
\mathbf{k}\right\rangle ^{\left\vert \alpha\right\vert -m}\Vert\Delta
_{\mathbf{k}}^{\alpha}\partial_{x}^{\beta}a(x,\mathbf{k}))\Vert_{L(E)}\text{,}%
\]
and in the case $r\in (0,\infty)\setminus\N$ we define
\begin{align*}
\left\Vert a\right\Vert _{m}^{\left(  \rho,r\right)  } & :=\left\Vert
a\right\Vert _{m}^{\left(  \rho,\left\lfloor r\right\rfloor \right)  }\\%
& +\max_{\substack{\left\vert \alpha\right\vert \leq\rho\\\left\vert
\beta\right\vert =\left\lfloor r\right\rfloor }}\sup_{\substack{x,y\in
\mathbb{T}^{n}\\x\neq y}}\sup_{\mathbf{k}\in\mathbb{Z}^{n}}\left\langle
\mathbf{k}\right\rangle ^{\left\vert \alpha\right\vert -m}\frac{\Vert
\Delta_{\mathbf{k}}^{\alpha}\partial_{x}^{\beta}a(x,\mathbf{k}))-\Delta
_{\mathbf{k}}^{\alpha}\partial_{y}^{\beta}a(y,\mathbf{k}))\Vert_{L(E)}%
}{\left\vert x-y\right\vert ^{r-\left\lfloor r\right\rfloor }}.
\end{align*}

b)  For $a\in S^{m,\rho,r}$  the pseudo-differential
operator $\op[a]$ is defined by%
\begin{equation}
  \label{2-1}
  (\op[a]f)(x) = \sum_{\bk\in\Z^n} e^{i\bk\cdot x} a(x,\bk) \hat f(\bk)\quad (f\in C^\infty(\T^n,E),\, x\in\T^n).
\end{equation}
\end{definition}

\begin{remark}
  \label{2.2}
  a) It is easily seen that for $f\in C^\infty(\T^n,E)$ we have $(\hat f(\bk))_{\bk\in\Z^n}\in\mathscr S(\Z^n,E)$, where $\mathscr S(\Z^n,E)$ stands for the Schwartz space of all functions $\phi\colon \Z^n\to E$ with $\sup_{\bk\in\Z^n} \langle \bk\rangle^N \|\phi(\bk)\|<\infty$ for all $N\in\N$ (see, e.g., \cite{Barraza-Denk-Hernandez_Nau16}, Lemma~2.2). Therefore, the sum in \eqref{2-1} converges absolutely.

  b) Inserting the definition of $\hat f(\bk)$ into the right-hand side of \eqref{2-1}, we formally get
  \begin{equation}
    \label{2-2}
    \begin{aligned}
      (\op[a]f)(x) & = \sum_{\bk \in\Z^n} \int_{\T^n} e^{i\bk\cdot(x-y)} a(x,\bk) f(y)\dbar y \\
      & = \sum_{\bk\in\Z^n}\int_{\T^n} e^{i\bk y} a(x,\bk) f(x-y)\dbar y.
    \end{aligned}
  \end{equation}
  However, this sum-integral does not converge in general. To make such integrals convergent (and to change the order of integration and summation), one has to use either oscillatory sum-integrals (see \cite{Barraza-Denk-Hernandez_Nau16}, Remark~3.4) or use integration by parts (see \cite{Ruzhansky-Turunen10}, Remark~4.1.18). In the cases considered below, the symbols will be good enough to guarantee absolute convergence of the sum-integrals.
\end{remark}

The definition of toroidal Besov spaces is based on a dyadic decomposition in the covariable space $\Z^n$. We use the following definition.

\begin{definition}
  \label{2.3} A sequence $(\varphi_j)_{j\in\N_0}\subset\mathscr S(\Z^n)$ is called a dyadic decomposition if the following conditions are satisfied.
  \begin{enumerate}
    [(i)]
    \item We have $\supp\varphi_0\subset\{\bk\in\Z^n: |\bk|\le 2\}$ and $\supp\varphi_j \subset\{\bk\in\Z^n: 2^{j-1}\le|\bk|\le 2^{j+1}\}$ for $j\in\N$.
    \item For each $\bk\in\Z^n$, we have $0\le \varphi_j(\bk)\le 1\;(j\in\N_0)$ and $\sum_{j\in\N_0} \varphi_j(\bk) = 1$.
    \item For each $\alpha\in\N_0^n$, exists a constant $c_\alpha>0$ independent of $j$ and $\bk$ such that
        \[ |\Delta_\bk^\alpha \varphi_j(\bk)| \le c_\alpha \langle \bk\rangle^{-|\alpha|}\quad (j\in\N,\, \bk\in\Z^n).\]
  \end{enumerate}
\end{definition}

\begin{remark}
  \label{2.4}
  A partition of unity on $\Z^n$ can be obtained as a restriction of a partition of unity on $\R^n$ in the sense of \cite{Barraza-Denk-Hernandez_Nau16}, Definition~3.5, or \cite{Amann97}, Section~4. Here, the definition of a partition of unity $(\tilde\varphi_j)_{j\in\N_0}$ on $\R^n$ includes the condition
  \[ |\partial_\xi^\alpha \tilde\varphi_j(\xi)|\le c_\alpha 2^{-j|\alpha|}\quad (\xi\in\R^n).\]
  Taking $\varphi_j:=\tilde\varphi_j|_{\Z^n}$, we obtain condition~\ref{2.3} (iii) by
  \cite{Ruzhansky-Turunen10}, proof of Theorem~II.4.5.3, which states that for each $\mathbf{k}%
\in\mathbb{Z}^{n}$,
\[
\Delta_{\mathbf{k}}^{\gamma}\varphi_{j}(\mathbf{k})=\partial_{\xi}^{\gamma
}\tilde\varphi_{j}(\xi)\big|_{\xi=\tilde{\xi}}%
\]
with some $\tilde{\xi}\in\lbrack k_{1},k_{1}+\gamma_{1}]\times\ldots
\times\lbrack k_{n},k_{n}+\gamma_{n}]$. This implies
\begin{equation}
|\Delta_{\mathbf{k}}^{\gamma}\varphi_{j}(\mathbf{k})|\leq  C\langle
\mathbf{k}\rangle^{-|\gamma|}\quad(j\in\mathbb{N},\,\mathbf{k}\in
\mathbb{Z}^{n}) \label{2-3}%
\end{equation}
using the conditions on the support of $\varphi_j$.
\end{remark}

Throughout the following, we will fix a dyadic decomposition $(\varphi_j)_{j\in\N_0}\subset\mathscr S(\Z^n)$. We set $\varphi_{-1}:=0$ and define
\[ \chi_j := \varphi_{j-1} + \varphi_j + \varphi_{j+1} \quad (j\in\N_0).\]
Then $\chi_j=1$ on $\supp\varphi_j$, i.e., we have $\varphi_j\chi_j = \varphi_j$ for all $j\in\N_0$.

\begin{definition}
  \label{2.5} For $p,q\in[1,\infty]$ and $s\in\R$, the Besov space $B_{pq}^s(\T^n,E)$ is defined as the space of all $u\in\mathscr D'(\T^n,E)$ with $\|u\|_{B_{pq}^s(\T^n,E)}<\infty$, where
  \[ \|u\|_{B_{pq}^s(\T^n,E)} := \Big\| \big( 2^{js} \|\op[\varphi_j]u\|_{L^p(\T^n,E)}\big)_{j\in\N_0}\Big\|_{\ell^q(\N_0)}.\]
\end{definition}

For properties of vector-valued Besov spaces on the torus, we refer to \cite{Barraza-Denk-Hernandez_Nau16}, Remark~3.9. For the analog spaces in $\R^n$, see \cite{Amann97}, Section~5. The Besov space does not depend on the choice of the dyadic decomposition (in the sense of equivalent norms).

The estimates for pseudodifferential operators on toroidal Besov spaces below are based on their representation as integral operators and estimates for their kernels. We adapt this representation to the dyadic decomposition and obtain better convergence properties. In particular, there is no need to consider oscillatory sum-integrals.

\begin{lemma}
  \label{2.6}
  Let $a\in S^{m,\rho,r}$, and let $f\in C^\infty(\T^n,E)$.

  a) We have
  \[ (\op[a]f)(x) = \sum_{\kappa\in\N_0} (\op[a]\op[\varphi_\kappa] f)(x)\quad (x\in\T^n).\]
  Here, the series on the right-hand side converges in $C(\T^n,E)$ (i.e., uniformly in $x$).

  b) For every $x\in\T^n$ and $j\in\N_0$,
  \[ \big(\op[a]\op[\varphi_j]f\big)(x) = \int_{\T^n} K_j(x,y)f(x-y)\dbar y ,\]
  where
  \begin{equation}
    \label{2-4}
    K_j(x,y) := \sum_{\bk\in\Z^n} e^{i\bk\cdot y} a(x,\bk) \varphi_j(\bk).
  \end{equation}
  (Note that this is a finite sum.)

  c) For every $x\in\T^n$ and $j\in\N_0$,
  \[ \big(\op[a]\op[\varphi_j] f\big)(x) = \sum_{\kappa\in\N_0}\Big[ \int_{\T^n}\int_{\T^n} K_{j\kappa}^{(1)} (x,y,z) (\op[\chi_\kappa]f)(x-y-z)\dbar y\dbar z\Big],\]
  where
  \[ K_{j\kappa}^{(1)} (x,y,z) := \sum_{\bk,\ell\in\Z^n} e^{i\ell\cdot y} e^{i\bk\cdot z} \varphi_j(\ell) \varphi_\kappa(\bk) a(x,\bk).\]

  d) For every $x\in\T^n$ and $j\in\N_0$,
  \[ \big(\op[\varphi_j] \op[a]f\big)(x) = \sum_{\kappa\in\N_0}\Big[ \int_{\T^n}\int_{\T^n} K_{j\kappa}^{(2)} (x,y,z) (\op[\chi_\kappa]f)(x-y-z)\dbar y\dbar z\Big],\]
  where
  \[ K_{j\kappa}^{(2)} (x,y,z) := \sum_{\bk,\ell\in\Z^n} e^{i\ell\cdot y} e^{i\bk\cdot z} \varphi_j(\ell) \varphi_\kappa(\bk) a(x-y,\bk).\]

The series over $\kappa$ in c) and d) converge in $C(\T^n,E)$, the sums over $\bk$ and $\ell$ are finite.
\end{lemma}

\begin{proof}
  a) Because of $\sum_{\kappa\in\N_0}\varphi_\kappa = 1$, we obtain
  \begin{equation}
    \label{2-5}
    (\op[a]f)(x) =\sum_{\bk\in\Z^n} e^{ix\cdot\bk} a(x,\bk) \hat f(k) = \sum_{\bk\in\Z^n} \Big( \sum_{\kappa\in\N_0} e^{ix\cdot\bk} a(x,\bk) \varphi_\kappa(\bk)\hat f(\bk)\Big).
  \end{equation}
  For every $\bk\in\Z^n$, there are at most three $\kappa\in\N_0$ with $\varphi_\kappa(\bk)\not=0$. This and $\varphi_\kappa\le 1$ yield
  \begin{align*}
    \sum_{\bk\in\Z^n,\,\kappa\in\N_0} \big\| & a(x,\bk)\varphi_\kappa(\bk)\hat f(\bk)\big\| \le 3\sum_{\bk\in\Z^n} \| a(x,\bk)\|_{L(E)} \|\hat f(\bk)\|_E \\
    & \le C \sum_{\bk\in\Z^n} \langle \bk\rangle^m \|\hat f(\bk)\| <\infty.
  \end{align*}
  In the last step, we have used $(\hat f(\bk))_{\bk\in\Z^n} \in \mathscr S(\Z^n,E)$. Therefore, the series in \eqref{2-5} converges in $C(\T^n,E)$, and we may change the order of summation which yields a).

  b) This follows from
  \begin{align*}
    (\op[a]\op[\varphi_j]f)(x) & = \int_{\bk\in\Z^n} e^{i\bk\cdot x} a(x,\bk) \varphi_j(\bk)\hat f(\bk) \\
    & = \sum_{\bk\in\Z^n}\Big[ \int_{\T^n} e^{i\bk\cdot x} a(x,\bk) \varphi_j(\bk) e^{-i\bk z} f(z)\dbar z\Big]\\
    & = \sum_{\bk\in\Z^n} \Big[\int_{\T^n} e^{i\bk\cdot y} a(x,\bk) \varphi_j(\bk) f(x-y)\dbar y\Big]\\
    & = \int_{\T^n} K_j(x,y) f(x-y)\dbar y.
  \end{align*}
  Note that the sum is finite, and therefore we may change the order of summation and integration.

  c) We use $\varphi_\kappa\chi_\kappa=\varphi_\kappa$ and $\op[\varphi_j]\op[\varphi_\kappa]=\op[\varphi_\kappa]\op[\varphi_j]$ and apply a) to get
  \[ \op[a]\op[\varphi_j]f = \sum_{\kappa\in\N_0}\op[a]\op[\varphi_\kappa]\op[\varphi_j]\op[\chi_\kappa] f.\]
  Here, the sum on the right-hand side converges in $C(\T^n,E)$ due to a).
  Applying b), we see that
  \[ \big(\op[a]\op[\varphi_\kappa]\op[\varphi_j]\op[\chi_\kappa]f\big)(x) = \int_{\T^n} K_\kappa(x,z)\big(\op[\varphi_j]\op[\chi_\kappa] f\big)(x-z)\dbar z\]
  with $K_\kappa$ being defined in \eqref{2-4}. Another application of b) with $a$ being replaced by the constant symbol $(x,\bk)\mapsto \id_E$ gives
  \[ \big(\op[\varphi_j]\op[\chi_\kappa]f\big)(x) = \int_{\T^n} \tilde K_j(y)\big(\op[\chi_\kappa]f\big)(x-y)\dbar y\]
  with $\tilde K_j(y) := \sum_{\ell\in\Z^n} e^{i\ell\cdot y} \varphi_j(\ell)$. Altogether we obtain
  \[ \big(\op[a]\op[\varphi_j]f\big)(x) = \sum_{\kappa\in\N_0} \int_{\T^n}\int_{\T^n} K_{j\kappa}^{(1)}(x,y,z)(\op[\chi_\kappa]f)(x-y-z)\dbar y\dbar z\]
  with
  \[ K_{j\kappa}^{(1)}(x,y,z) := K_\kappa(x,z) \tilde K_j(y).\]

  d) Similarly, we apply a) and twice b) to get
  \begin{align*}
    \big( \op[\varphi_j]&\op[a]f\big) (x) = \sum_{\kappa\in\N_0} \big(\op[\varphi_j]\op[a]\op[\varphi_\kappa]\op[\chi_\kappa]f\big)(x) \\
    & = \sum_{\kappa\in\N_0} \int_{\T^n}\int_{\T^n} K_{j\kappa}^{(2)} (x,y,z)(\op[\chi_\kappa]f)(x-y-z)\dbar y\dbar z,
  \end{align*}
  with $K_{j\kappa}^{(2)} (x,y,z) := \tilde K_j(y) K_\kappa(x-y,z)$ which shows the assertion in d).
\end{proof}

The following estimate on the kernel $K_j$ defined in Lemma~\ref{2.6} will be one key ingredient for the proof of Besov space continuity of toroidal pseudodifferential operators.

\begin{theorem}
  \label{2.7}
  Let $b\in S^{m,\rho,0}$, and set
  \[ K_j(x,y) := \sum_{\bk\in\Z^n} e^{i\bk\cdot y}\varphi_j(\bk)b(x,\bk)\quad (j\in\N_0).\]
  Then
  \begin{equation}
    \label{2-6}
    \| K_j(x,y)\|_{L(E)} \le C \, 2^{jm} g_{j,\theta}(y) \|b\|_{m}^{(\rho,0)} \quad (x,y\in\T^n,\, j\in\N_0,\,\theta\in(0,1)),
  \end{equation}
  where
  \[ g_{j,\theta}(y) := \frac{(2^j|y|)^\theta}{|y|^n(1+2^j|y|)}\quad (y\in\T^n).\]
\end{theorem}

\begin{proof}
  The proof follows the ideas from \cite{Barraza-Denk-Hernandez_Nau16}, proof of Lemma~4.8.

  Note that $\varphi_{j}(\mathbf{k})=0$ for $|\mathbf{k}|>2^{j+1}$ implies
$\Delta_{\mathbf{k}}^{\gamma}\varphi_{j}(\mathbf{k})=0$ for $|\mathbf{k}%
|>2^{j+1}+|\gamma|$. In the same way, $\varphi_{j}(\mathbf{k})=0$ for
$|\mathbf{k}|<2^{j-1}$ implies $\Delta_{\mathbf{k}}^{\gamma}\varphi
_{j}(\mathbf{k})=0$ for $|\mathbf{k}|<2^{j-1}-|\gamma|$.

Let $n_{0}$ be the smallest integer such that $2^{-n_{0}}(n+1)\leq\frac{1}{4}%
$. Then
\[
2^{j+1}+|\gamma|\leq2^{j+1}+(n+1)\leq2\cdot2^{j+1}=2^{j+2}%
\]
and
\[
2^{j-1}-|\gamma|\geq2^{j-1}-(n+1)\geq\frac{1}{2}\cdot2^{j-1}=2^{j-2}%
\]
hold for all $j\geq n_{0}$ and all $|\gamma|\leq n+1$.

Condition \ref{2.3} (iii) and the condition $a\in S^{m,\rho,0}$ imply with
the discrete Leibniz formula that
\[
\Vert\Delta_{\mathbf{k}}^{\gamma}(\varphi_{j}(\mathbf{k})a(x,\mathbf{k}%
))\Vert_{L(E)}\leq C\left\Vert a\right\Vert _{m}^{\left(  \rho,0\right)
}\langle\mathbf{k}\rangle^{m-|\gamma|}
\]
 for $\left(  x,\mathbf{k}\right)\in\mathbb{T}^{n}\times\mathbb{Z}^{n}$, $j\in\mathbb{N}_{0}$ and $|\gamma|\leq n+1$. Moreover, for each $x\in\mathbb{T}^{n}$ and $j\geq n_{0}$ we have%
\begin{equation}
\Delta_{\mathbf{k}}^{\gamma}(\varphi_{j}(\mathbf{k})a(x,\mathbf{k}))=0
\label{ec supp operator differ phijmala(k)}%
\end{equation}
if $|\mathbf{k}|<2^{j-2}$ or if $|\mathbf{k}|>2^{j+2}$.

Let $N\in\{n,n+1\}$, and set $(e^{i\eta}-1)^\gamma := \prod_{k=1}^n (e^{i\eta_k}-1)^{\gamma_k}$. Then we have (see \cite{Barraza-Denk-Hernandez_Nau16}, Remark~4.7)
\[
|\eta|^{N}\leq C\sum_{|\gamma|=N}\big|(e^{-i\eta}-1)^{\gamma}\big|\quad
(\eta\in\mathbb{T}^{n})
\]
and
\[
(e^{-i\eta}-1)^{\gamma}K_{j}(x,\eta)=\sum_{\mathbf{k}\in\mathbb{Z}^{n}%
}(e^{i\mathbf{k}\cdot\eta}-1)\Delta_{\mathbf{k}}^{\gamma}(\varphi
_{j}(\mathbf{k})a(x,\mathbf{k})).
\]
In combination with the elementary inequality  $|e^{i\bk\cdot\eta}-1|\le 2|\bk|^\theta|\eta|^\theta$ which holds for all $\theta\in(0,1)$, we get
\begin{equation}
|\eta|^{N}\Vert K_{j}(x,\eta)\Vert_{L(E)}\leq C\left\Vert a\right\Vert
_{m}^{\left(  \rho,0\right)  }|\eta|^{\theta}\sum_{\mathbf{k}\in B_{j}%
}|\mathbf{k}|^{\theta}\langle\mathbf{k}\rangle^{m-N}\quad(x,\eta
\in\mathbb{T}^{n}) \label{0-2}%
\end{equation}
with $B_{j}:=\{\mathbf{k}\in\mathbb{Z}^{n}:2^{j-2}\leq|\mathbf{k}|\leq
2^{j+2}\}$. Due to \cite{Barraza-Denk-Hernandez_Nau16}, inequality (4-5), for all $\mu>0$ the
inequality
\[
\sum_{\substack{\ell\in\mu^{-1}\mathbb{Z}^{n}\setminus\{0\}\\|\ell|_{\infty
}\leq1}}|\ell|^{\theta-n}\mu^{-n}\leq C_{\theta}%
\]
holds. Setting $\mu:=2^{j+2}$, we obtain
\begin{align*}
\sum_{\substack{\mathbf{k}\in\mathbb{Z}^{n}\setminus\{0\}\\|\mathbf{k}%
|\leq2^{j+2}}}|\mathbf{k}|^{\theta-n}  &  =\sum_{\substack{\ell\in\mu
^{-1}\mathbb{Z}^{n}\setminus\{0\}\\|\ell|\leq1}}|\mu\ell|^{\theta-n}\leq
\sum_{\substack{\ell\in\mu^{-1}\mathbb{Z}^{n}\setminus\{0\}\\|\ell|_{\infty
}\leq1}}|\mu\ell|^{\theta-n}\\
&  \leq C_{\theta}\mu^{\theta}=C2^{j\theta}.
\end{align*}
Inserting this into \eqref{0-2} with $N=n$ yields
\begin{equation}
\sum_{\mathbf{k}\in B_{j}}|\mathbf{k}|^{\theta}\langle\mathbf{k}\rangle
^{m-n}\leq\Big(\sup_{\mathbf{k}\in B_{j}}\langle\mathbf{k}\rangle^{m}%
\Big)\sum_{\mathbf{k}\in B_{j}}|\mathbf{k}|^{\theta-n}\leq C\cdot
2^{j(m+\theta)}. \label{ec estimate with i=0}%
\end{equation}
Note here that for $m\geq0$ we used the estimate
$$\langle\mathbf{k}\rangle^{m}\leq C\cdot2^{(j+2)m}=C\cdot2^{jm},$$
while for $m<0$ we used
$$\langle\mathbf{k}\rangle^{m}\leq C\cdot2^{(j-2)m}=C\cdot2^{jm}.$$

For \eqref{0-2} with $N=n+1$ we have in the same way
\begin{equation}
\sum_{\mathbf{k}\in B_{j}}|\mathbf{k}|^{\theta}\langle\mathbf{k}%
\rangle^{m-n-1}\leq C\sum_{\mathbf{k}\in B_{j}}|\mathbf{k}|^{\theta-n}%
\langle\mathbf{k}\rangle^{m-1}\leq C\cdot2^{j(\theta+m-1)}.
\label{ec estimate with i=1}%
\end{equation}
Therefore, we obtain
\begin{align*}
|\eta|^{n}\Vert K_{j}(x,\eta)\Vert_{L(E)}  &  \leq C\left\Vert a\right\Vert
_{m}^{\left(  \rho,0\right)  }\cdot2^{j(m+\theta)}|\eta|^{\theta},\\
|\eta|^{n+1}\Vert K_{j}(x,\eta)\Vert_{L(E)}  &  \leq C\left\Vert a\right\Vert
_{m}^{\left(  \rho,0\right)  }\cdot2^{j(m+\theta-1)}|\eta|^{\theta}.
\end{align*}
Multiplying the second inequality by $2^{j}$ and adding both inequalities
yields
\[
\Vert K_{j}(x,\eta)\Vert_{L(E)}\leq C\left\Vert a\right\Vert _{m}^{\left(
\rho,0\right)  }\cdot2^{jm}\,\frac{(2^{j}|\eta|)^{\theta}}{|\eta|^{n}%
(1+2^{j}|\eta|)}\quad(x,\eta\in\mathbb{T}^{n},\,j\geq n_{0}).
\]
\end{proof}

\section{Mapping properties in toroidal Besov spaces}

In this section, we use the kernel estimates from above to show continuity of pseudodifferential operators in toroidal vector-valued Besov spaces.

\begin{lemma}
  \label{3.1}
  a) Let $p\in[1,\infty]$, and let $K\colon\T^n\times\T^n\to L(E)$ be measurable. Assume that there exists a function $g\in L^1(\T^n)$ with
  \[
  \|K(x,y)\|_{L(E)}\le g(y)\;(x,y\in\T^n).
  \]
  For $x\in\T^n$ and $f\in L^p(\T^n,E)$, define $F(x) := \int_{\T^n} K(x,y)f(x-y)\dbar y$. Then $F(x)$ is well-defined for almost all $x\in\T^n$ and
  \[ \| F\|_{L^p(\T^n,E)} \le \|g\|_{L^1(\T^n)} \|f\|_{L^p(\T^n,E)}\quad (f\in L^p(\T^n,E)).\]

  b) Let $p\in[1,\infty]$, let $K\colon\T^n\times\T^n\times \T^n\to L(E)$ be measurable. Assume that there exist functions $g, h\in L^1(\T^n)$ with
  \[ \|K(x,y,z)\|_{L(E)}\le g(y)h(z)\;(x,y,z\in\T^n).\]
   For $x\in\T^n$, define $F(x) := \int_{\T^n}\int_{\T^n} K(x,y,z)f(x-y-z)\dbar y\dbar z$. Then $F(x)$ is well-defined for almost all $s\in\T^n$ and
  \[ \| F\|_{L^p(\T^n,E)} \le \|g\|_{L^1(\T^n)}\|h\|_{L^1(\T^n)} \|f\|_{L^p(\T^n,E)}\quad (f\in L^p(\T^n,E)).\]
\end{lemma}

\begin{proof}
  a) Let $p\in[1,\infty)$. For $x\in\T^n$, we have
  \begin{align*}
    \|F(x)\| & = \Big\| \int_{\T^n} K(x,y)f(x-y)\dbar y\Big\| \le \int_{\T^n} \|K(x,y)\|_{L(E)} \|f(x-y)\|\dbar y\\
    & \le \int_{\T^n} g(y) \|f(x-y)\| \dbar y = (g * \|f\|)(x).
  \end{align*}
  Therefore,
  \[ \|F\|_{L^p(\T^n,E)} \le \big\| g * \|f\|\,\big\|_{L^p(\T^n)} \le \|g\|_{L^1(\T^n)}\|f\|_{L^p(\T^n,E)}.\]
  In particular, this yields that $F(x)$ is well-defined for almost all $x\in\T^n$. The case $p=\infty$ follows similarly.

  b) This follows in the same way. By the assumption on $K$, we can estimate
  \[ \|F(x)\| \le \int_{\T^n}\Big[\int_{\T^n} g(y)h(z)\|f(x-y-z)\| \dbar y\Big]\dbar z = \big( h* (g*\|f\|)\big)(x).\]
  This yields the desired estimate on $\|F\|_{L^p(\T^n,E)}$ and the fact that $F(x)$ is well-defined for almost all $x\in\T^n$.
\end{proof}

\begin{lemma}
  \label{3.2}
  Let $a\in S^{m,\rho,r}$ with $r\in (0,1)$.

  a) For all $j\in\N_0$ and $f\in C^\infty(\T^n,E)$,
  \[ \|\op[a]\op[\varphi_j]f\|_{L^p(\T^n,E)} \le C \|a\|_m^{(\rho,r)} 2^{jm} \|\op[\chi_j]f\|_{L^p(\T^n,E)}.\]

  b) For all $j\in\N_0$ and $f\in C^\infty(\T^n,E)$,
  \begin{align*}
   \|\op[\varphi_j] & \op[a]f\|_{L^p(\T^n,E)} \\
   & \le C \|a\|_m^{(\rho,r)} \big(2^{jm} \|\op[\chi_j]f\|_{L^p(\T^n,E)} + 2^{-jr} \|f\|_{B_{p1}^m(\T^n,E)}\big).
   \end{align*}
  \end{lemma}

  \begin{proof}
    a) By Lemma~\ref{2.6} b),
    \[ (\op[a]\op[\varphi_j]f)(x) = \int_{\T^n} K_j(x,y)f(x-y)\dbar y\]
    with $K_j$ being defined in \eqref{2-4}. Due to Theorem~\ref{2.7}, for arbitrary $\theta\in (0,1)$,
    \[ \|K_j(x,y)\|_{L(E)} \le C 2^{jm} \|a\|_m^{(\rho,r)} g_{j,\theta}(y)\quad (x,y\in\T^n).\]
    Because of
    \begin{align*}
      \|g_{j,\theta}\|_{L^1(\T^n)} & = \int_{\T^n} \frac{(2^j|y|)^\theta}{|y|^n(1+2^j|y|)}\dbar y = \int_{2^j\T^n} \frac{|z|^\theta}{|z|^n(1+|z|)}\dbar z  \\
      & \le \int_{\R^n}  \frac{|z|^\theta}{|z|^n(1+|z|)}\dbar z  <\infty,
    \end{align*}
    we can apply Lemma~\ref{3.1} a) to obtain the assertion of a).

    b) We consider the difference
    \begin{align*}
      \big( \op[\varphi_j]\op[a] & - \op[a]\op[\varphi_j]\big)f(x) \\
      & = \sum_{\kappa\in\N_0} \int_{\T^n}\int_{\T^n} K_{j\kappa}(x,y,z) (\op[\chi_\kappa]f)(x-y-z)\dbar y\dbar z
    \end{align*}
    with
    \begin{align*}
      K_{j\kappa}&(x,y,z) := K_{j\kappa}^{(2)}(x,y,z) - K_{j\kappa}^{(1)}(x,y,z) \\
      & = \sum_{\bk,\ell\in\Z^n} e^{i\ell\cdot y} e^{i\bk\cdot z} \varphi_j(\ell)\varphi_\kappa(\bk)\big( a(x-y,\bk)-a(x,\bk)\big)\\
      & = \Big(\sum_{\ell\in\Z^n} e^{i\ell\cdot y}\varphi_j(\ell)\Big) \Big(\sum_{\bk\in\Z^n} e^{i\bk \cdot z}\varphi_\kappa(\bk) \big( a(x-y,\bk)-a(x,\bk)\big)\Big)\\
      & =: \tilde K_j(y) K_\kappa'(x,y,z).
    \end{align*}
    We apply Theorem~\ref{2.7} with $b(x,\bk):= a(x-y,\bk)-a(x,\bk)$ where $y\in\T^n$ is fixed. By the definition of $S^{m,\rho,r}$ we have
    \begin{align*}
     \|b\|_m^{(\rho,0)} & = \max_{|\alpha|\le \rho} \sup_{x\in\T^n}\sup_{\bk\in\Z^n} \langle\bk\rangle^{|\alpha|-m} \| \Delta_\bk^\alpha (a(x-y,\bk)-a(x,\bk))\|_{L(E)}\\
     & \le |y|^r \|a\|_{m}^{(\rho,r)}.
     \end{align*}
     Note that $0<r<1$. From Theorem~\ref{2.7}  we get
     \[ \|K_\kappa'(x,y,z)\|_{L(E)} \le C |y|^r 2^{\kappa m} \|a\|_m^{(\rho,r)} g_{\kappa,\theta_1}(z)\quad (x,y,z\in\T^n)\]
     for arbitrary $\theta_1\in (0,1)$. Another application of Theorem~\ref{2.7} with constant symbol $b(x,\bk) = \id_E$ yields
     \[ \|\tilde K_j(y)\|_{L(E)} \le C g_{j,\theta_2}(y)\quad (y\in\T^n)\]
     for all $\theta_2\in (0,1)$. Therefore,
     \[ \|K_{j\kappa}(x,y,z)\|_{L(E)} \le C 2^{\kappa m} \|a\|_{m}^{(\rho,r)} |y|^r g_{\kappa,\theta_1}(z) g_{j,\theta_2}(y).\]
     Because of $r\in (0,1)$, we can choose $\theta_2\in (0,1-r)$ and obtain for $\theta_0 := \theta_2+r \in (0,1)$
     \[ |y|^r g_{j,\theta_2}(y) = |y|^r\,\frac{(2^j|y|)^{\theta_2}}{|y|^n(1+2^j|y|)} = 2^{-jr} g_{j,\theta_0}(y).\]
     Therefore,
     \[ K_{j\kappa}(x,y,z)\|_{L(E)} \le C 2^{\kappa m} 2^{-jr} \|a\|_m^{(\rho,r)} g_{j,\theta_0}(y)g_{\kappa,\theta_1}(z).\]
     We have seen above that $\|g_{j,\theta_0}\|_{L^1(\T^n)}\le C<\infty$ and $\|g_{\kappa,\theta_1}\|_{L^1(\T^n)} \le C<\infty$. Therefore, we can apply Lemma~\ref{3.1} b) to get
     \begin{align*}
       \big\| \big(\op[\varphi_j]\op[a] & - \op[a]\op[\varphi_j]\big)f\big\|_{L^p(\T^n,E)} \\
       & \le C 2^{-jr} \|a\|_m^{(\rho,r)} \sum_{\kappa\in\N_0} 2^{\kappa m} \|[\op\chi_\kappa]f\|_{L^p(\T^n,E)}.
     \end{align*}
     By the definition of $\chi_\kappa$,
     \begin{align*}
       \sum_{\kappa\in\N_0} 2^{\kappa m} & \|\op[\chi_\kappa]f\|_{L^p(\T^n,E)} \\
       &=  \sum_{\kappa\in\N_0} 2^{\kappa m}  \big( \|(\op[\varphi_{\kappa-1}]+ \op[\varphi_{\kappa}]+\op[\varphi_{\kappa+1}])f\|_{L^p(\T^n,E)}\\
       & \le (2^{-m} +1 + 2^m) \sum_{\kappa\in\N_0} 2^{\kappa m} \|\op[\varphi_\kappa]f\|_{L^p(\T^n,E)}\\
       & = C \|f\|_{B_{p1}^m(\T^n,E)}.
     \end{align*}
     Therefore,
     \[ \|(\op[\varphi_j]\op[a]-\op[a]\op[\varphi_j])f\|_{L^p(\T^n,E)} \le C 2^{-jr} \|a\|_m^{(\rho,r)} \|f\|_{B_{p1}^m(\T^n,E)}.\]
     Together with part a) this yields the assertion of b).
  \end{proof}

  The last lemma is the essential step in the proof of Besov space continuity. The following theorem is  the main result of the present paper.

  \begin{theorem}\label{3.3}
  Let $m\in\R$, $\rho\in\N$ with $\rho\ge n+1$, and $r\in (0,\infty)$, and let $a\in S^{m,\rho,r}$. Then for $s\in (0,r)$ and $p,q\in [1,\infty]$, the mapping
\[
\op[a]\colon B_{pq}^{s+m}(\mathbb{T}^{n},E)\rightarrow B_{pq}^{s}%
(\mathbb{T}^{n},E)
\]
is continuous. Moreover,
\[
\left(  a\mapsto \op[a]\right)  \in L\left(  S^{m,\rho,r},L\left(
B_{pq}^{s+m}(\mathbb{T}^{n},E),B_{pq}^{s}(\mathbb{T}^{n},E)\right)  \right)  .
\]
\end{theorem}

\begin{proof}
\textbf{(i)} We first consider the case $r\in (0,1)$.
We start with showing that
\[
\op[a]\colon B_{p1}^{s+m}(\mathbb{T}^{n},E)\to B_{p1}^{s}%
(\mathbb{T}^{n},E)
\]
is continuous. For that we will use the density of $C^{\infty}(\mathbb{T}^{n},E)$ in $B_{p1}^{s+m}(\mathbb{T}^{n},E)$ (see \cite{Barraza-Denk-Hernandez_Nau16}, Theorem 3.15). Let $f\in C^{\infty}(\mathbb{T}^{n},E)$. Then by Lemma \ref{3.2} b) we obtain
that
\begin{align*}
&  \big\Vert \op[a]f\big\Vert_{B_{p1}^{s}(\mathbb{T}^{n},E)} =\sum_{j=0}^{\infty}2^{js}\Vert \op[\left.  \varphi_{j}\right\vert
_{\mathbb{Z}^{n}}]\op[a]f\Vert_{L^{p}(\mathbb{T}^{n},E)}\\
&  \leq C\left\Vert a\right\Vert _{m}^{\left(  \rho,r\right)  }\bigg(
\sum_{j \in\N_0}2^{j\left(  s+m\right)  }\Vert\op[\chi_{j}%
]f\Vert_{L^{p}(\mathbb{T}^{n},E)}+\Vert f\Vert_{B_{p1}%
^{m}(\mathbb{T}^{n},E)}\sum_{j \in\N_0}2^{j\left(  s-r\right)  }\bigg).
\end{align*}
We have seen in the proof of Lemma~\ref{3.2} that the first sum can be estimated by $C\|f\|_{B_{p1}^{s+m}(\T^n,E)}$. For the second term, we note that $\sum_{j\in\N_0} 2^{j(s-r)}$ is finite because of $r>s$ and use the continuous embedding $B_{p1}^{s+m}\left(  \mathbb{T}^{n},E\right)  \hookrightarrow B_{p1}^{m}\left(
\mathbb{T}^{n},E\right)$. Therefore,
\[ \|\op[a] f\|_{B_{p1}^s(\T^n,E)} \le C \|a\|_m^{(\rho,r)} \|f\|_{B_{p1}^{s+m}(\T^n,E)} \]
which shows the continuity of $\op[a]\colon  B_{p1}^{s+m}(\mathbb{T}^{n},E)\to B_{p1}^{s}(\mathbb{T}^{n},E)$ as well as the continuity of $a\mapsto \op[a]$ for $q=1$.

For general
$q\in\left[  1,\infty\right]  $ we use  real interpolation theory: For
$q\in\left[  1,\infty\right]  $, we choose some $0<\epsilon<1$ such that
$s-\epsilon,s+\epsilon\in\left(  0,r\right)  .$ Then%
\[
B_{pq}^{t}(\mathbb{T}^{n},E)=\left(  B_{p1}^{t-\epsilon}(\mathbb{T}%
^{n},E),B_{p1}^{t+\epsilon}(\mathbb{T}^{n},E)\right)  _{1/2,q}\text{   for
}t\in\left\{  s,s+m\right\}  .
\]
Now the continuity of%
\[
\op[a]\colon B_{p1}^{s\pm\epsilon+m}(\mathbb{T}^{n},E)\to B_{p1}^{s\pm\epsilon}(\mathbb{T}^{n},E)
\]
and  real interpolation  immediately give the
continuity of%
\[
\op[a]\colon B_{pq}^{s+m}(\mathbb{T}^{n},E)\to B_{pq}^{s}(\mathbb{T}^{n},E).
\]
In the same way, the continuity of the map $a\mapsto \op[a]$ follows.

\medskip

\textbf{(ii)} Now let $r\in [1,\infty)$, and let $s\in (0,r)$. We first assume that $s\not\in \N$, i.e., $s=s_0+s_1$ with $s_0\in\N$ and $s_1\in (0,1)$. We choose $\tilde r\in (s,r]$ such that $r_1:= \tilde r-s_0\in (0,1)$. Then $a\in S^{m,\rho,\tilde r}$ by the definition of the symbol class.

We make use of an equivalent norm in $B_{pq}^s(\T^n,E)$. More precisely, there exist constants $c_1,c_2>0$ such that
\begin{equation}\label{3-1}
 c_1 \sum_{|\alpha|\le s_0} \|\partial_x^\alpha u\|_{B_{pq}^{s_1}(\T^n,E)} \le \|u\|_{B_{pq}^s(\T^n,E)} \le c_2 \sum_{|\alpha|\le s_0} \|\partial_x^\alpha u\|_{B_{pq}^{s_1}(\T^n,E)}
 \end{equation}
for all $u\in B_{pq}^s(\T^n,E)$, see \cite{Amann97}, (5.19), for the case of $\R^n$, and \cite{Arendt-Bu04}, proof of Theorem 2.3, for the one-dimensional torus.

Let $f\in C^\infty(\T^n,E)$, and let $j\in\N_0$ and  $\alpha\in\N_0^n$ with $|\alpha|\le s_0$. By Lemma~\ref{2.6} d) and the Leibniz rule, we have
\begin{align*}
&\op[\varphi_j]\partial_x^\alpha\op[a]f =\partial_x^\alpha\op[\varphi_j]\op[a]f \\
& = \partial_x^\alpha\Big[ \sum_{\kappa\in\N_0} \int_{\T^n}\int_{\T^n} K_{j\kappa}^{(2)} (x,y,z) (\op[\chi_\kappa]f)(x-y-z)\dbar y\dbar z\Big]\\
& = \sum_{\beta\le\alpha} \sum_{\kappa\in\N_0} \int_{\T^n}\int_{\T^n}
(\partial_x^\beta K_{j\kappa}^{(2)})(x,y,z) \big( \op[\chi_\kappa]\partial_x^{\alpha-\beta}f\big) (x-y-z)\dbar y\dbar z\\
& = \sum_{\beta\le \alpha} \binom\alpha\beta\big( \op[\varphi_j]\op[a_\beta] (\partial_x^{\alpha-\beta} f)\big)(x)
\end{align*}
with the symbol $a_\beta(x,\bk) := (\partial_x^\beta a)(x,\bk)\;(x\in\T^n,\,\bk\in\Z^n)$. Here we note that for all $|\beta|\le s_0$, we have $a_\beta\in S^{m,\rho,r_0}$ with $\|a_\beta\|_m^{(\rho,r_0)} \le C \|a\|_m^{(\rho,\tilde r)}\le C \|a\|_m^{(\rho,r)}$. In particular, the series over $\kappa$ above are uniformly convergent with respect to $x$ by Lemma~\ref{2.6} d) and we may change the order of differentiation and integration.

For $|\alpha|\le s_0$ and $\beta \le \alpha$, we can apply part (i) of the proof and obtain
\begin{align*}
 \|\op[a_\beta]\partial_x^{\alpha-\beta} f\|_{B_{pq}^{s_1}(\T^n,E)} & \le C \|a_\beta\|_m^{(\rho,r_0)} \|\partial_x^{\alpha-\beta}f\|_{B_{pq}^{s_1+m}(\T^n,E)}\\
& \le \|a\|_m^{(\rho,r)} \|f\|_{B_{pq}^{s+m}(\T^n,E)}.
\end{align*}
Together with \eqref{3-1}, this yields
\[ \|\op[a]f\|_{B_{pq}^s(\T^n,E)} \le c_2 \sum_{|\alpha|\le s_0} \|\partial_x^\alpha \op[a] f\|_{B_{pq}^{s_1}(\T^n,E)} \le C \|a\|_m^{(\rho,r)} \|f\|_{B_{pq}^{s+m}(\T^n,E)}.\]
This shows the desired continuity in the case $s\in (0,r)\setminus\N$. Finally, if $s\in\N$, we choose $\epsilon\in (0,1)$ with $0<s-\epsilon<s+\epsilon<r$. As we have seen,
\[ \op[a]\colon B^{s+m\pm\epsilon}_{pq}(\T^n,E)\to B^{s\pm\epsilon}_{pq}(\T^n,E)\]
is continuous. Now the continuity of
\[ \op[a]\colon B^{s+m }_{pq}(\T^n,E)\to B^{s }_{pq}(\T^n,E)\]
again follows by real interpolation $(\ldots)_{1/2,q}$. So we have seen that the continuity of the operator $\op[a]$ stated in the theorem as well as the continuity of $a\mapsto \op[a]$ hold in all cases.
\end{proof}

\begin{remark}
  \label{3.4}
  a) As a particular case, we obtain the continuity of $\op[a]$ in the case of $x$-independent symbols. In fact, this could more easily be obtained by the observation that  $\op[\varphi_j]\op[a] = \op[a]\op[\varphi_j]$ holds in this case. Therefore, one can apply Lemma~\ref{2.6} b) and Lemma~\ref{3.2} a) and avoid double integrals.

  The case of $x$-independent symbols was already shown in \cite{Barraza-Denk-Hernandez_Nau16}, Theorem~3.17. However, the proof in \cite{Barraza-Denk-Hernandez_Nau16} was based on the connection between the symbols on $\Z^n$ and the symbols on $\R^n$. In fact, every symbol on $\Z^n$ can be extended to a symbol on $\R^n$ belonging to the same symbol class (see  \cite{Ruzhansky-Turunen10}, Theorem~II.4.5.3, and  the transference principle in \cite{Hytoenen-van_Neerven-Veraar-Weis16}, Section~5.7). In the present paper, we formulated a proof which is independent of this fact. Therefore, the present proof might serve as a basis for generalizations to more general groups instead of $\T^n$.

  b) As the symbols considered here are of restricted smoothness, we do not obtain continuity in the full scale of Besov spaces. That the range of continuity is restricted becomes obvious if we take a symbol $a(x,\bk)=b(x)$ independent of $\bk$, where $b\in C^r(\T^n,L(E))$. In this case,
  $a\in S^{0,\rho,r}$ and
   \[
(\op[a]f)(x)=\sum_{\mathbf{k}\in\mathbb{Z}^{n}}\int_{\mathbb{T}^{n}%
}e^{i\mathbf{k}\cdot y}b(x)f(x-y)\dbar y=b(x)f(x) \quad (x\in\T^n).
\]
Taking $f(x)$ as a constant function, we see that in general $\op[a]f\in C^r(\T^n,E)$ cannot be improved.
\end{remark}

\def\cprime{$'$}

\end{document}